\documentclass{amsart}
\usepackage{hyperref}
\usepackage{graphics}
\usepackage{color}
\usepackage{amssymb}
\usepackage[all]{xy}
\usepackage{amsmath}
\usepackage{stmaryrd}
\CompileMatrices
\newtheorem{theorem}{Theorem}[section]
\newtheorem{lemma}[theorem]{Lemma}
\newtheorem{proposition}[theorem]{Proposition}
\newtheorem{corollary}[theorem]{Corollary}

\theoremstyle{definition}

\newtheorem{example}[theorem]{Example}

\newtheorem{remark}[theorem]{Remark}
\theoremstyle{remark}

\numberwithin{equation}{section}
\def\CC{{\mathbb C}}
\def\KK{{\mathbb K}}

\def\ZZ{{\mathbb Z}}

\def\QQ{{\mathbb Q}}

\begin{document}
\title[Factorially graded rings and Cox rings]%
{Factorially graded rings and Cox rings}
\author[B.~Bechtold]{Benjamin Bechtold} 
\address{Mathematisches Institut, Universit\"at T\"ubingen,
Auf der Morgenstelle 10, 72076 T\"ubingen, Germany}
\email{bechtold@mail.mathematik.uni-tuebingen.de}

\begin{abstract}
Cox rings of normal varieties are factorially graded, i.e. homogeneous elements allow a unique decomposition into homogeneous factors. We study this property from an algebraic point of view and give a criterion which in a sense reduces it to factoriality. This will allow us to detect and construct Cox rings in a purely algebraic manner. 
\end{abstract}


\maketitle

\section{Introduction}

Throughout this article, $K$ is an abelian group and $R$ a $K$-graded algebra without $K$-homogeneous zero divisors over a noetherian ring $A$. By $R^+$ we denote the multiplicative monoid of $K$-homogeneous elements of $R$.  Following \cite{Ar, Ha2}, we say that a non-zero $f \in R^+\setminus R^*$ is {\em $K$-prime\/} if $f | gh$ with $g,h \in R^+$ implies $f | g$ or $f | h$. The ring $R$ is said to be {\em factorially $K$-graded\/} if every non-zero $f \in R^+ \setminus R^*$ is a product of $K$-primes. In general, factorially graded rings need not be UFDs \cite{Ar, HaHe} and it is only in the case of torsion free grading groups that these notions are equivalent \cite{An}. 

Most notably, Cox rings of normal varieties over an algebraically closed field $\KK$ of characteristic zero are always factorially graded with respect to the natural grading by the class group. In fact, Cox rings are characterized in algebraic terms as follows  \cite{Ar, ArDeHaLa, BeHa, Ha2}: A finitely generated algebra $R$ graded by a finitely generated abelian group $K$ is a Cox ring of some $\QQ$-factorial projective variety if and only if it is $(i)$ normal , $(ii)$ factorially graded, $(iii)$ has only trivial homogeneous units  and is $(iv)$ {\em almost freely graded\/}, i.e. for a system $f_1,\ldots,f_m \in R$ of pairwise non-associated $K$-prime generators each $m-1$ of their degrees generate $K$. 

Our aim is to provide criteria that will enable us to detect and construct Cox rings. The first result states that graded factoriality may be reduced to factoriality in the following sense:

\begin{theorem}\label{thm:alg-criterion-for-graded-factoriality-intro}
 Let $R = A[f_1,\ldots, f_m]$ a $K$-graded ring with $K$-prime generators $f_1,\ldots,f_m$ and let $S$ be the multiplicative system generated by the $f_i$. Then the following are equivalent: 
\begin{enumerate}
 \item $R$ is factorially $K$-graded,
 \item $(S^{-1}R)_0$ is factorial. 
\end{enumerate}
 
\end{theorem}

This allows generalization in the following way. Let $K'$ a subgroup of $K$. The Veronese subalgebra of $R$ corresponding to $K'$ is denoted $R_{K'}$. If any $K$-homogeneous element in $R$ is associated to a homogeneous element in $R_{K'}$, then $R$ is called {\em $K'$-associated\/}. Moreover, we say that $R$ is {\em $K$-noetherian\/} if any $K$-homogeneous ideal is generated by finitely many $K$-homogeneous elements. Now, our result states as follows (see Section 2 for the proof):

\begin{theorem}\label{thm:alg-criterion-for-graded-factoriality}
  Let $R$ a $K$-graded $K$-noetherian ring. Then for every subgroup $K'$ of $K$ the following are equivalent:
\begin{enumerate}
 \item $R$ is factorially $K$-graded,
 \item there is a multiplicative system $S$ generated by $K$-primes such that
\begin{enumerate}
 \item $S^{-1}R$ is $K'$-associated, 
 \item $(S^{-1}R)_{K'}$ is factorially $K'$-graded. 
\end{enumerate} 
\end{enumerate}
Furthermore, if $R$ is factorially $K$-graded and $K'$ a subgroup of $K$ then for any $S$ generated by $K$-primes such that $S^{-1}R$ is $K'$-associated, the Veronese subalgebra $(S^{-1}R)_{K'}$ is factorially $K'$-graded. 
\end{theorem}

Our second tool for the construction of Cox rings turns a given factorially graded ring into a an almost freely and factorially graded ring. It is the homogeneous version of Gau{\ss}' Theorem, which is proven in Section 3:

\begin{theorem}
\label{thm:homgauss}
Let $R$ a $K$-graded ring without homogeneous zero divisors and let $R[T]$ be $K$-graded through a choice of a $K$-degree for $T$. Then $R$ is factorially $K$-graded if and only if $R[T]$ is so.
\end{theorem}

Another way to ascertain that $R$ is almost freely graded is to coarsen the grading. By the following theorem (proven in Section 4), graded factoriality is preserved if the coarsening consists in dropping a free direct summand of the grading group:

\begin{theorem}\label{thm:variation-of-free-part-of-grading}
 Let $R$ be a $K \oplus \ZZ^m$-graded $K$-noetherian ring without $K \oplus \ZZ^m$-homogeneous zero divisors. Then $R$ is factorially $K \oplus \ZZ^m$-graded if and only if it is factorially $K$-graded. Furthermore, if $R$ is factorially graded, then a $K \oplus \ZZ^m$-homogeneous element has one and the same decomposition with respect to both gradings. 
\end{theorem}

In Section 5 we give examples of Cox rings constructed via the tools presented above. Among others, we treat the rings $\CC[T_1,\ldots,T_4]/\langle T_1^{m_1} + \ldots + T_4^{m_4} \rangle$ whose divisor class groups are calculated in \cite{St}. It is an immediate consequence of our results that all those rings, UFDs and non-UFDs, are factorially graded and the adjunction of a further homogeneous variable will then turn these rings into Cox rings. 
The class of rings that our results allow us to verify algebraically as Cox rings includes all Cox rings calculated geometrically in \cite{De, HaTs, HaSu}.

The author wishes to thank J\"urgen Hausen for many fruitful discussions. 

\section{Proof of Theorem~\ref{thm:alg-criterion-for-graded-factoriality}}

A non-zero $f \in R^+ \setminus R^*$ is {\em $K$-irreducible\/} if $f = g h$ with $g, h \in R^+$ implies $g \in R^*$ or $h \in R^*$. We now study the behaviour of factorially graded rings with respect to localizations. 
Let $R = \oplus_{w \in K}{R_w}$ be a $K$-graded ring and $S\subset R^+$ a multiplicative system. Then the localization $S^{-1}R$ is $K$-graded via 
\[ (S^{-1}R)_w := \left\{ \frac{r}{s} \in S^{-1}R; \; r \in R_{\deg_K(s) +w} \right\} \text{ for }w \in K.
 \]
Note that the localization $S^{-1}\mathfrak{a}$ of a $K$-homogeneous ideal is again $K$-homogeneous.

\begin{proposition}\label{prop:localising-preserves-graded-factoriality}
Let $R$ a $K$-graded ring and $S \subset R^+$ a multiplicative system. If $R$ is factorially $K$-graded then so is $S^{-1}R$. 
\end{proposition}
\begin{proof}
 In general, for a $K$-prime $f \in R$ the fraction $f / 1 \in S^{-1}R$ is either $K$-prime or a unit, depending on whether or not $\langle f \rangle \cap S$ is empty: If $\langle f \rangle \cap S = \emptyset$ and $f/1$ divides a product of homogeneous elements $g/s, h/t \in S^{-1}R$, then there are $r \in R^+$ and $q \in S$ with $g h q = f s t r \in \langle f \rangle$. Since $q$ cannot lie in $\langle f \rangle$ either $g$ or $h$ must, in particular $f/1$ divides either $g/1$ or $h/1$ in $S^{-1}R$. If otherwise $\langle f \rangle \cap S$ is non-empty then there is an $r \in R$ with $r f= s \in S$, so $(f/1) (r/s) = 1$. 

Now, if $R$ is factorially $K$-graded, we may write an arbitrary $f/s$ as $(1/s) (f/1)$ and decompose $f$ into $K$-prime factors $f_1\cdots f_m$. Grouping the fractions $f_i/1$ that are units together with $1/s$ we get the desired decomposition. 
\end{proof}

The converse of the above is true if $R$ is $K$-noetherian and $S$ is generated by $K$-primes, as the following two statements show, the latter of which was formulated by Nagata \cite{Na} with respect to localizations of UFDs.

\begin{lemma}\label{lem:nagata0}
 Let $R$ a $K$-graded $K$-noetherian ring. Then every homogeneous non-zero non-unit is a product of $K$-irreducible elements. 
\begin{proof}
Suppose that the set $M$ of principal ideals $\langle r \rangle$ generated by elements $r \in R^+\setminus \{0\}$ who are no products of $K$-irreducible elements is non-empty. Then, by $K$-noetherianity it has a maximal element $\langle r' \rangle$ whose generator $r'$ is in particular not $K$-irreducible. So there are $s, t \in R^+ \setminus R^*$ with $r' = s t$ and $\langle r' \rangle \subsetneq \langle s \rangle, \langle t \rangle$ are proper inclusions. Thus, by maximality of $\langle r' \rangle$ the elements $s$ and $t$ are products of $K$-irreducible elements. But then, so is $r'$ - a contradiction. 
\end{proof}
\end{lemma}

\begin{lemma}\label{lem:nagata1}
 Let $R$ be a $K$-graded ring such that every homogeneous non-zero non-unit is a product of $K$-irreducible elements. Let $S$ a multiplicative system generated by $K$-prime elements of $R$. If $S^{-1}R$ is factorially $K$-graded then so is $R$. 
\begin{proof}
By Lemma~\ref{lem:nagata0}, we have show that every $K$-irreducible $f \in R$ is $K$-prime. Let $G$ be a set of $K$-prime generators of $S$. If $f$ is divided by an element of $G$, then it is already associated to that element and hence also $K$-prime. So assume that no element of $G$ divides $f$. Then $f/1 \in S^{-1}R$ is no unit and is in fact $K$-irreducible: For if $f s_1 s_2 = r_1 r_2$ with $r_i\in R^+$ and $s_i \in S$, the $K$-prime factors of $s_1 s_2$ can be factored out of $r_1 r_2$, leaving a decomposition $f = r_1' r_2'$. But then $r_1' \in R^*$ or $r_2' \in R^*$, so $r_1 \in (S^{-1}R)^*$ or $r_2 \in (S^{-1}R)^*$. Since $S^{-1}R$ is factorially $K$-graded, we can conclude that $f/1$ is $K$-prime in $S^{-1}R$.

Now consider $g, h \in R^+$ with $f | g h$. By $K$-primality of $f/1$ we may assume $f | g$ in $S^{-1}R$, so there exist $s \in S$ and $r \in R^+$ with $r f = s g$. Since the $K$-prime factors of $s$ all have to occur in $r$, we get $r 'f = g$ for some $r' \in R$, so $f | g$. 
\end{proof}

\end{lemma}

The above preliminaries show that factorially graded rings behave with respect to graded localizations analogously to UFDs. However, the following lemma, which is essential for our proof of Theorem~\ref{thm:alg-criterion-for-graded-factoriality}, is a priori unique to the setting of graded rings. For it relies on the fact that if a product $f g$ of homogeneous elements and one of its factors $f$ lie in the Veronese subalgebra $R_{K'}$, then so does the other factor $g$. An analogous statement on subrings of a non-graded ring is generally false. 
 
\begin{lemma}\label{main-lemma}
Let $R$ be a $K$-graded ring and $K'$ a subgroup of $K$ such that $R$ is $K'$-associated. Then the following hold:
\begin{enumerate}
 \item A element $f \in R_{K'}^+$ is $K'$-prime in $R_{K'}$ if and only if it is $K$-prime in $R$.
\item The ring $R$ is factorially $K$-graded if and only if $R_{K'}$ is factorially $K'$-graded.
\end{enumerate}
\end{lemma}
\begin{proof}
First, note that for homogeneous $f, g$ with $f \in R_{K'}$ and $f g \in R_{K'}$ we indeed have $\deg_K(g) + \deg_K(f) \in K'$ and $\deg_K(f) \in K'$, so $\deg_K(g) \in K'$, i.e. $g \in R_{K'}$. 

For $(i)$, consider a $K'$-prime $f \in R_{K'}$. Given $g, h, t \in R^+$ with $f t = g h$, we fix $s_g, s_h \in R^* \cap R^+$ with $g s_g, h s_h \in R_{K'}$ and get $f (t s_g s_h) = (g s_g)(h s_h)$. Then $t s_g s_h$ already lies in $R_{K'}$, so $f | g s_g$ or $f | h s_h$ in $R_{K'}$, which implies $f | g$ or $f | h$ in $R$. Thus, $f$ is $K$-prime. Conversely, a $K$-prime $f \in R_{K'}$ clearly is $K'$-prime in $R_{K'}$. 

For $(ii)$, first suppose that $R$ is factorially $K$-graded and let $f \in R_{K'}^+\setminus R^*$. Then $f$ has a decomposition $f = f_1 \cdots f_m$ in $R$. Taking $s_i \in R^* \cap R^+$ with $f_i s_i \in R_{K'}$ we get $f = (s_2^{-1} \cdots s_m^{-1} f_1) (f_2 s_2) \cdots (f_m s_m)$ with all factors in brackets lying in $R_{K'}$. By $(i)$, this is a decomposition into $K'$-primes as needed. For the converse, let $R_{K'}$ be factorially $K'$-graded and $f \in R^+ \setminus R^*$. Again, we fix $s_f \in R^* \cap R^+$ with $f s_f \in R_{K'}$. Then $f s_f$ has a decomposition into $K'$-primes which are also $K$-primes in $R$ by $(i)$, and multiplication by $s_f^{-1}$ yields a decomposition of $f$ into $K$-primes.  
\end{proof}

\begin{proof}[Proof of Theorem~\ref{thm:alg-criterion-for-graded-factoriality}]
Let $R$ be factorially $K$-graded and let $K'$ be a subgroup of $K$. The existence of an $S$ such that $S^{-1}R$ is $K'$-associated is obvious: take $S:=R^+ \setminus \{0\}$. Now let $S$ be any such  multiplicative system. By Proposition~\ref{prop:localising-preserves-graded-factoriality} the localization $S^{-1}R$ is factorially $K$-graded and by choice of $S$, we may apply Lemma~\ref{main-lemma} and conclude that $(S^{-1}R)_{K'}$ is factorially $K'$-graded.

Conversely, suppose there is an $S$ generated by $K$-primes and a subgroup $K'$ such that $S^{-1}R$ is $K'$-associated and $(S^{-1}R)_{K'}$ is factorially $K'$-graded. Then Lemma~\ref{main-lemma} yields that $S^{-1}R$ is factorially $K$-graded. And since $R$ is $K$-noetherian, Lemmas~\ref{lem:nagata0} and ~\ref{lem:nagata1} imply that $R$ is factorially $K$-graded.
\end{proof}

\section{Proof of Theorem~\ref{thm:homgauss}}

In this section, let $R$ be a $K$-graded ring (without $K$-homogeneous zero divisors) and let $R[T]$ have the $K$-grading extending that of $R$ which is obtained through the choice of an arbitrary $w \in K$ as the $K$-degree of $T$. As before,  the $K$-degree of a $K$-homogeneous polynomial $f \in R[T]$ is denoted $\deg_K(f)$, whereas $\deg(f) \in \ZZ_{\geq 0}$ is the maximal exponent occuring in $f$ with a non-vanishing coefficient, as usual.  

Since $T$ is $K$-homogeneous, a $K$-homogeneous polynomial $f \in R[T]$ has only $K$-homogenous coefficients. 
So, as $R$ has no $K$-homogeneous zero divisors the leading term of a product $g h$ with $g, h \in R[T]^+$ is just the product of the leading terms of $g$ and $h$, i.e. $R[T]$ has no $K$-homogeneous zero divisors either.  

\begin{lemma}\label{lem-proof-homgauss-if} 
Let $R$ and $R[T]$ be $K$-graded as above. Then a non-zero $p \in R^+\setminus R^*$ is $K$-prime if and only if it is $K$-prime in $R[T]$. 
Thus, if $R$ is factorially graded then the decomposition of $r \in R$ into $K$-prime factors in $R$ is also a decomposition into $K$-prime factors in $R[T]$. 
\begin{proof}
By the above, $R/\langle p \rangle$ has no homogeneous zero divisors precisely if $(R/\langle p \rangle)[T]$, which is isomorphic to $R[T]/\langle p\rangle$, has none. 
\end{proof}
\end{lemma}

\begin{proof}[Proof of Theorem~\ref{thm:homgauss}, ``if''-direction]
 Let $R[T]$ be factorially $K$-graded and let $r \in R$. Then all $K$-prime factors of the decomposition of $r$ in $R[T]$ are constant and therefore $K$-prime in $R$. So $R$ is factorially $K$-graded. 
\end{proof}

Next, we prepare the proof of the converse. 

\begin{proposition}
For $R$ and $R[T]$ as above the following are equivalent: 
\begin{enumerate}
 \item $R^+\setminus \{0\} \subseteq R^*$,
 \item every $K$-homogeneous ideal $\mathfrak{a}$ of $R[T]$ is generated by a $K$-homogeneous element $f \in \mathfrak{a}$ of minimal degree $\deg(f) \in \ZZ_{\geq 0}$. 
\end{enumerate}
\begin{proof}
Let $R^+ \setminus \{0\} \subseteq R^*$. First, we observe that $R[T]$ allows division with remainder for $K$-homogeneous elements in the following sense: For non-zero polynomials $f,g \in R[T]^+$ there are $q ,r \in R[T]^+$ with $f = q g + r$ such that $r =0$ or $\deg(r) < \deg(g)$. 
So let 
\[f = a_0 T^0 +\ldots +a_m T^m, \quad g = b_0 T^0 +\ldots +b_n T^n\in R[T]^+ .\] 
 If $m=0$, then $f = 0 \cdot g + f$ is as wanted. Let now $m>0$. We only need to consider the case $m \geq n$. Then $f':= f - b_n^{-1}a_m T^{m-n} g$ is $K$-homogeneous of degree $\deg_K(f)$ and by induction we find $q', r \in R[T]^+$ with $f'= q^{\prime} g + r$ and $r =0$ or $\deg(r) < \deg(g)$. Thus, we get $f = q g +r$ where $q:=q^{\prime} + b_n^{-1}a_m T^{m-n}$. 

Now consider a $K$-homogeneous ideal $\mathfrak{a} \unlhd R[T]$. Let $f \in \mathfrak{a} \cap R[T]^+$ have minimal degree $\deg(f) \in \ZZ_{\geq 0}$. For $g \in \mathfrak{a} \cap R[T]^+$ there are $q, r \in R[T]^+$ with $g = q f + r$ and $\deg(r) < \deg(f)$ or $r=0$. Minimality of $\deg(f)$ implies $r=0$. 
\end{proof}

\end{proposition}

\begin{proposition}
Let $R$ a $K$-graded ring such that every $K$-homogeneous ideal is generated by a single homogeneous element. Then $R$ ist factorially graded. 
\begin{proof}
 Since $R$ is $K$-noetherian, Lemma~\ref{lem:nagata0} reduces the problem to showing that every $K$-irreducible $p \in R$ is $K$-prime. By definition, $\langle p \rangle$ is maximal among all the principal ideals of $K$-homogeneous elements - so in our case among all homogeneous ideals. Thus, in $R \langle p \rangle$ every $K$-homogeneous element is a unit, i.e. no zero divisor. Hence, $p$ is $K$-prime. 
\end{proof}
\end{proposition}

\begin{corollary}\label{cor34}
 Let $R$ be a $K$-graded ring. Then $(R^+\setminus \{0\})^{-1}R[T]$ is factorially $K$-graded. 
\end{corollary}

Now let $R$ be factorially $K$-graded and let $P$ be a system of $K$-prime elements such that every $K$-prime element of $R$ is associated to exactly one $p \in P$. For brevity set $R':=(R^+\setminus \{0\})^{-1}R$. We write $\nu_p(f)$ for the minimal multiplicity with which $p \in P$ occurs in the decompositions of the coefficients of $f \in R'[T]^+$. 

\begin{lemma}\label{gauss-lemma}
In the above notation, for any $p \in P$ the multiplicities of two $K$-homogeneous polynomials $f, g \in R'[T]$ satisfy
\[ \nu_p(fg) = \nu_p(f) + \nu_p(g) \text{ .}\]
Furthermore, a polynomial $f \in R[T]^+$ with $\nu_p(f)=0$ for all $p \in P$ is $K$-prime in $R[T]$ if and only if it is $K$-prime in $R'[T]$. 
\begin{proof}
Let $f, g \in R'[T]^+$. The proclaimed equation clearly holds if $f$ or $g$ are constant. So we only treat the case that $\nu_p(f)=\nu_p(g) = 0$ for all $p \in P$. In particular, $f$ and $g$ are polynomials over $R$. Now for all $p \in P$, the classes of $f$ and $g$ in $(R/\langle p \rangle) [T]$ are non-zero and therefore their product $f g$ is non-zero as well, because $R / \langle p \rangle$ and thereby $ (R/\langle p \rangle) [T]$ has no $K$-homogeneous zero divisors. But this is equivalent to $\nu_p(f g)=0$ for all $p \in P$, which we had to show. 

Now, let $f \in R[T]$ satisfy $\nu_f(p)=0$ for all $p \in P$ and let $f$ be $K$-prime in $R'[T]$. If $f$ divides in $R[T]$ the product of two $K$-homogeneous $g, h \in R[T]$, it must divide either $g$ or $h$ in $R'[T]$. So we may assume that there is a $q \in R'[T]$ with $f q = g$. Applying the above, we get $\nu_p(q) = \nu_p(f q) = \nu_p(g) \geq 0$ for all $p \in P$, so $q \in R[T]$ and $f | g$ in $R[T]$. Thus, $f$ is $K$-prime in $R[T]$. We omit the proof of the converse, as it is not needed in the following.
\end{proof}
\end{lemma}

\begin{proof}[Proof of Theorem~\ref{thm:homgauss}, ``only if''-direction]
Let $R$ be factorially $K$-graded and $f \in R[T]^+ \setminus R[T]^*$. By Corollary~\ref{cor34}, $R'[T]$ is factorially $K$-graded. Thus, $f$ has a decomposition $f=f_1 \cdots f_m$ with $K$-prime elements $f_i \in R'[T]$. Multiplication with a suitable fraction of $K$-homogeneous elements of $R$ yields $f = c f_1' \cdots f_m'$ where the $f_i'$ lie in $R[T]$ and satisfy $\nu_p(f_i')=0$ for all $p \in P$. By Lemma~\ref{gauss-lemma}, the $f_i'$ are $K$-prime in $R[T]$ and $c \in R$ because $\nu_p(c) = \nu_p(f)\geq 0$ for all $p \in P$. Since the decomposition of $c$ in $R$ is also a decomposition in $R[T]$, the proof is complete. 
\end{proof}

\section{Proof of Theorem~\ref{thm:variation-of-free-part-of-grading}}

\begin{proposition}
Let $R$ be a $K \oplus \ZZ^m$-graded ring without $K \oplus \ZZ^m$-homogeneous zero divisors and $f \in R$ a $K \oplus \ZZ^m$-homogeneous element. Then $f$ is $K \oplus \ZZ^m$-prime if and only if it is $K$-prime. 
\end{proposition}
\begin{proof}
We only show the ``only if''-direction for $m=1$. Let $f$ be $K \oplus \ZZ$-prime. Let $g$ and $h$ be $K$-homogeneous with $f | g h$. Let $g= g_{n_1}+\ldots+g_{n_2}$ and $h = h_{m_1} + \ldots + h_{m_2}$ be the decompositions into $\ZZ$-homogeneous parts. Since $R$ has no $K \oplus \ZZ$-homogeneous zero divisors, $g_{n_2} h_{m_2}$ is the leading term of $g h$, so $\ZZ$-homogeneity of $\mathfrak{a}$ yields $g_{n_2} h_{m_2} \in \langle f \rangle$, so $g_{n_2} \in \langle f \rangle$ or $h_{m_2} \in \langle f \rangle$ by $K \oplus \ZZ$-primality of $f$. Now we have $(g- g_{n_2})h \in \langle f \rangle$ or $g(h-h_{m_2}) \in \langle f \rangle$ and by induction on the added number of $\ZZ$-homogeneous parts of $g$ and $h$, we get $f | g$ or $f|h$.
\end{proof}

\begin{proof}[Proof of Theorem~\ref{thm:variation-of-free-part-of-grading}]
 Let $R$ be factorially $K \oplus \ZZ^m$-graded. By Theorem~\ref{thm:alg-criterion-for-graded-factoriality}, there is a multiplicative system $S$ generated by $K \oplus \ZZ^m$-primes such that $S^{-1}R$ is $\{0_K\} \oplus \ZZ^m$-associated and $(S^{-1}R)_{\{0_K\} \oplus \ZZ^m}$ is factorially graded, i.e. factorial. With respect to the coarsened $K$-grading, this is just the Veronese subalgebra $(S^{-1}R)_{0_K}$. Since the generators of $S$ are $K$-prime by the above proposition and $S^{-1}R$ is clearly $\{0_K\}$-associated, Theorem~\ref{thm:alg-criterion-for-graded-factoriality} implies that $R$ is factorially $K$-graded. 

Conversely, let $R$ be factorially $K$-graded and let $f$ be a non-zero $K \oplus \ZZ^m$-homogeneous non-unit. Then all $K$-prime factors in the decomposition of $f$ are $K \oplus \ZZ$-homogeneous and therefore $K \oplus \ZZ$-prime by the above proposition. 
\end{proof}

\section{Examples}

In the following, let $\KK$ be algebraically closed of characteristic zero and $K$ finitely generated. In order to give examples of factorially graded rings and Cox rings, we first fix the notation for a more explicit version of our criterion for graded factoriality. 
Let $R := \KK[T_1,\ldots,T_m] / \mathfrak{a}$ an integral $K$-graded $\KK$-algebra with $K$-prime generators $\overline{T_i}$ and $R^+ \cap R^* = \KK^*$. Let $Q\colon \ZZ^m \rightarrow K$ denote the map given by $e_i \mapsto \deg_K(\overline{T_i})$. For a free subgroup $K'$ of $K$, we fix a monomorphism $B: \ZZ^n \rightarrow \ZZ^m$ mapping onto $Q^{-1}(K')$ and denote by $\beta \colon \KK[T_1^{\pm 1},\ldots, T_n^{\pm 1}] \rightarrow \KK[T_1^{\pm 1},\ldots, T_m^{\pm 1}]$ the corresponding homomorphism of group algebras which is an isomorphism onto the Veronese subalgebra $\KK[T_1^{\pm 1},\ldots, T_m^{\pm 1}]_{K'}$. Let $g_1,\ldots,g_d$ be $K$-homogeneous generators of $\mathfrak{a}$ and $\mu_j \in \ZZ^m$ with $Q(\mu_j) - \deg_K(g_j) \in K'$. We denote by $h_j$ the unique preimage of $T^{-\mu_j}g_j \in \KK[T_1^{\pm 1},\ldots, T_m^{\pm 1}]_{K'}$ under $\beta$.

\begin{theorem}
 In the above notation, $R$ is factorially $K$-graded if and only if $\KK[T_1^{\pm 1},\ldots, T_n^{\pm 1}] / \langle h_1,\ldots, h_d \rangle$ is factorial. 
Moreover, if one of the above holds and $R$ is normal and almost freely graded, then $R$ is a Cox ring.
\end{theorem}

\begin{proof}
 Let $S$ be the multiplicative system given by products of the $\overline{T_i}$. Then $S^{-1}R$ is $K'$-associated and by Theorem~\ref{thm:alg-criterion-for-graded-factoriality} we only have to show that $(S^{-1}R)_{K'}$ is isomorphic to $\KK[T_1^{\pm 1},\ldots, T_n^{\pm 1}] / \langle h_1,\ldots, h_d \rangle$ which we see as follows:
 \begin{align*}
 (S^{-1}R)_{K'} &\cong \left(\KK[T_1^{\pm 1},\ldots, T_m^{\pm 1}] / \mathfrak{a}_T\right)_{K'} \\
&\cong \KK[T_1^{\pm 1},\ldots, T_m^{\pm 1}]_{K'} / (\mathfrak{a}_T \cap \KK[T_1^{\pm 1},\ldots, T_m^{\pm 1}]_{K'}) \\
&\cong \KK[T_1^{\pm 1},\ldots, T_n^{\pm 1}] / \beta^{-1}(\mathfrak{a}_T \cap \KK[T_1^{\pm 1},\ldots, T_m^{\pm 1}]_{K'}) .
\end{align*}
Here, $\mathfrak{a}_T$ denotes the localization of $\mathfrak{a}$ by $T_1 \cdots T_m$. By construction, the $h_j$ generate $\beta^{-1}(\mathfrak{a}_T \cap \KK[T_1^{\pm 1},\ldots, T_m^{\pm 1}]_{K'})$, which concludes the proof. 
\end{proof}

\begin{corollary}\label{cor:aff-lin-criterion}
 If in the above notation those exponent vectors of $T^{-\mu_j} g_j$ which are non-zero form a basis of a primitive sublattice of $Q^{-1}(K')$, then $R$ is factorially $K$-graded. 
Moreover, if $R$ is also normal and almost freely graded, then $R$ is a Cox ring. 
\begin{proof}
 After a unimodular transformation of the variables, the $h_j$ are affine linear polynomials. So $\KK[T_1^{\pm 1},\ldots, T_n^{\pm 1}]/\langle h_1,\ldots, h_d \rangle$ is as a graded ring isomorphic to some Laurent algebra $\KK[T_1^{\pm 1},\ldots, T_s^{\pm 1}]$, in particular it is factorially $K'$-graded, and we can apply the above theorem.  
\end{proof}
\end{corollary}
 
\begin{example}
 Consider the normal ring $R=\CC[T_1,\ldots, T_4] /\langle  g \rangle$ with $g = T_1^{m_1} + \ldots + T_4^{m_4}$. These rings were discussed by Storch, who in \cite{St} gave a criterion which of these are factorial and which are not, depending on the numbers $\gcd(m_i, m_j)$. However, all of them are factorially graded as we will now see: Let  $K:= \ZZ^4 / M$  with $M := \rm{span}_{\ZZ}((-m_1, m_2, 0, 0), (-m_1, 0, m_3, 0), (-m_1, 0, 0, m_4))$ and let $T_i$ have the degree $\deg_K(T_i):= e_i + M$. Then $g$ is $K$-homogeneous, so $R$ is $K$-graded. The $\overline{T_i} \in R$ are prime, i.e. $K$-prime. Since the non-zero exponent vectors of $T_1^{-m_1}g$ by definition form a basis of $\ker(Q)$, $R$ is factorially $K$-graded by Corollary~\ref{cor:aff-lin-criterion}. 
\end{example}
Still, not all of the above rings are Cox rings yet, but using the homogeneous Gau{\ss} Theorem we do get Cox rings:

\begin{example}
In the notation of 5.3, let $m_1 = m_2 = 2$ and $m_3 = m_4 = 3$. According to \cite{St}, $R$ is no UFD. The grading by $K = \ZZ \oplus \ZZ / 6 \ZZ$ is given via $\deg_K(\overline{T_1}) = (3, \overline{0}), \deg_K(\overline{T_2})= (3,\overline{3}), \deg_K(\overline{T_3}) = (2, \overline{0})$ and $\deg_K(\overline{T_4})= (2, \overline{2})$. Since $R$ allows a positive $\ZZ$-grading, it has only trivial units, and the $\overline{T_i}$ are pairwise non-associated. But no three of their degrees generate $K$, so $R$ is no Cox ring. However, by adding another variable $T_5$ with degree e.g. $(2 ,\overline{1})$ we do get a Cox ring $R[T_5]$ since the homogeneous Gau\ss\ Theorem ensures that graded factoriality is preserved. 
\end{example}

Next, we consider an example of an application of the Coarsening Theorem:

\begin{example}
 Let $R:=\KK[T_1,\ldots,T_5] / \langle g \rangle$ with $g:=T_1^2 + T_2^2T_3^3 + T_4^3 + T_5^4$. By Serre's criterion, $R$ is normal and the $\overline{T_i}$ are prime. Also, $R$ has only trivial units. $R$ is graded by $\tilde{K}:=\ZZ^2\oplus \ZZ/2\ZZ$ respectively $K:= \ZZ/2\ZZ$ via the matrices
\[\tilde{Q}:=\left(\begin{array}{c c c c c}
 0 & -3 & 2 & 0& 0 \\
 6 & 6 & 0 & 4& 3 \\
 \overline{1} & \overline{1} &\overline{0} &\overline{0} &\overline{0}
\end{array}\right), \quad 
Q:=\left(\begin{array}{c c c c c}
 6 & 3 & 2 & 4& 3 \\
 \overline{1} & \overline{1} &\overline{0} &\overline{0} &\overline{0}
\end{array}\right), \]
where the $i$-th column denotes $\deg_{\tilde{K}}(\overline{T_i})$ respectively $\deg_K(\overline{T_i})$. Since the non-zero exponents of $T_1^{-2}g$ form a basis of $\ker(\tilde{Q})$, $R$ is factorially $\tilde{K}$-graded by Corollary~\ref{cor:aff-lin-criterion}. However, the $\tilde{K}$-grading is not almost free, so as a $\tilde{K}$-graded ring, $R$ is no Cox ring. But by Theorem~\ref{thm:variation-of-free-part-of-grading}, $R$ is also factorially $K$-graded, and this grading clearly is almost free. So as a $K$-graded ring, $R$ is a Cox ring. 
\end{example}

\begin{remark}
In general, we may use Corollary~\ref{cor:aff-lin-criterion} to show that $R$ is factorially graded, if the monomials $q_{i,j}$ occuring in the defining relations $g_i$ each comprise a variable that occurs in no other $q_{l,k}$. The Cox rings calculated in \cite{De, HaTs, HaSu} are all members of this class of rings. The following example shows that the above method is not restricted to these cases: 
\end{remark}

\begin{example}
 Let $R=\CC[T_1,\ldots,T_4] / \langle g \rangle$ with $g = T_1^7 + T_1^2T_2T_4 + T_2^3T_3 + T_4^2$ be graded by $K:=\ZZ$ via $\deg_K(\overline{T_1})=2, \deg_K(\overline{T_2}) = 3, \deg_K(\overline{T_3}) = 5$ and $\deg_K(\overline{T_4}) = 7$, and let $Q: \ZZ^4\rightarrow K$ denote the corresponding linear map. Since the grading is pointed, $R$ has only trivial units. The $\overline{T_i}$ are prime and pairwise non-associated. Since those exponent vectors of $T_2^{-3}T_3^{-1} g$ which are non-zero form a basis of $\ker(Q)$, $R$ is factorial by Corollary~\ref{cor:aff-lin-criterion}. Hence, $R$ is a Cox ring. 
\end{example}

\end{document}